\documentclass[12pt]{amsart}

\usepackage{amssymb, amsmath, amsthm}
\theoremstyle{plain}
\newtheorem{thm}{Theorem}[section]

\newtheorem{lem}[thm]{Lemma}
\newtheorem{cor}[thm]{Corollary}
\theoremstyle{definition}
\newtheorem{defn}[thm]{Definition}

\theoremstyle{definition}

\newcommand{\C}{\mathbb{C}}

\newcommand{\Nzero}{\mathbb{N}\cup \{0\}}
\newcommand{\R}{\mathbb{R}}

\newcommand{\Z}{\mathbb{Z}}

\newcommand{\calB}{\mathcal{B}}

\newcommand{\calF}{\mathcal{F}}

\newcommand{\calM}{\mathcal{M}}
\newcommand{\calS}{\mathcal{S}}
\newcommand{\ichi}{\mathbf{1}}
\newcommand{\supp}{{\,\mathrm{supp}\,}}

\newcommand{\vs}{\vspace{6pt}}

\numberwithin{equation}{section}
\begin{document}

\title
{Notes on lattice bump Fourier multiplier operators 
on $L^2 \times L^2$} 

\author[T. Kato]{Tomoya Kato}
\author[A. Miyachi]{Akihiko Miyachi}
\author[N. Tomita]{Naohito Tomita}

\address[T. Kato]
{Division of Pure and Applied Science, 
Faculty of Science and Technology, Gunma University, 
Kiryu, Gunma 376-8515, Japan}

\address[A. Miyachi]
{Department of Mathematics, 
Tokyo Woman's Christian University, 
Zempukuji, Suginami-ku, Tokyo 167-8585, Japan}

\address[N. Tomita]
{Department of Mathematics, 
Graduate School of Science, Osaka University, 
Toyonaka, Osaka 560-0043, Japan}

\email[T. Kato]{t.katou@gunma-u.ac.jp}
\email[A. Miyachi]{miyachi@lab.twcu.ac.jp}
\email[N. Tomita]{tomita@math.sci.osaka-u.ac.jp}

\date{\today}

\begin{abstract}
Given a smooth bump function,
we consider the multiplier formed by taking
the linear combination of the translations
of the bump function and the corresponding
bilinear Fourier multiplier operator.
Under certain condition on the bump function,
we give a complete characterization of
the coefficients of the linear combination
for which the corresponding bilinear operator
defines a bounded operator
from $L^2\times L^2$ to $L^2$-based amalgam spaces.
\end{abstract}


\maketitle

\section{Introduction}
\subsection{Background}

For $\sigma \in L^{\infty}(\R^n \times \R^n)$, 
the bilinear Fourier multiplier 
operator $T_{\sigma}$ is defined by 
\begin{align*}
&T_{\sigma}(f, g)(x)
=
\iint_{\R^n \times \R^n} 
e^{2\pi i x \cdot (\xi + \eta)} 
\sigma (\xi, \eta) 
\widehat{f}(\xi)
\widehat{g}(\eta)\, 
d\xi d\eta, 
\\
& 
f, g \in \calS (\R^n), \; 
x \in \R^n.  
\end{align*}
We also write 
$T[\sigma]$ or $T[\sigma (\xi, \eta)]$
to denote 
$T_{\sigma}$. 
If $X, Y, Z$ are function spaces on $\R^n$ equipped with 
norms $\|\cdot \|_{X}, \|\cdot \|_{Y}, \|\cdot \|_{Z}$ respectively, 
then we define 
\begin{align*}
&
\|\sigma \|_{\calM (X \times Y \to Z)}
=
\|T_{\sigma}\|_{X \times Y \to Z}
\\
&=
\sup \{
\| T_{\sigma}(f,g)\|_{Z} 
\mid 
f \in \calS \cap X, \;
g \in \calS \cap Y, \;
\|f\|_{X}=\|g\|_{Y}=1
\}. 
\end{align*}
If $\|T_{\sigma}\|_{X \times Y \to Z}<\infty$, then, 
with a slight abuse of terminology,  
we shall say that $T_{\sigma}$ is bounded from $X\times Y$ to 
$Z$ and write 
$T_{\sigma}: X\times Y \to Z$.


One of the fundamental theorems for the boundedness of 
bilinear operators $T_{\sigma}$ in Lebesgue spaces 
is that if the multiplier $\sigma$ satisfies the condition 
\[
|\partial_{\xi}^{\alpha} \partial_{\eta}^{\beta} \sigma (\xi, \eta)|
\le C_{\alpha, \beta} (|\xi|+ |\eta|)^{-|\alpha| - |\beta|}
\]
then $T_{\sigma}: L^{p}\times L^q \to L^r$ 
for all $1<p, q\le \infty$ with $1/p + 1/q=1/r>0$. 
This theorem was proved by Coifman-Meyer \cite{CM}, 
Kenig-Stein \cite{Kenig-Stein}, and 
Grafakos-Torres \cite{GT}. 
This theorem may be considered as a natural extension 
of the well-known theorem about the linear Fourier multiplier 
operators satisfying H\"ormander-Mihlin type condition. 
However, if we consider 
other type of conditions, 
bilinear Fourier multiplier operators take properties 
that are different from the linear operators. 
One of such properties concerns 
with multipliers satisfying the condition 
\begin{equation}\label{S00}
|\partial_{\xi}^{\alpha} \partial_{\eta}^{\beta} \sigma (\xi, \eta)|
\le C_{\alpha, \beta}. 
\end{equation}
For the linear case, 
Plancherel's theorem simply
implies that the 
Fourier multiplier operator 
$f \mapsto \calF^{-1}(m\widehat{f})$ 
is bounded from $L^2$ to $L^2$ whenever 
$m\in L^{\infty}$. 
But for the bilinear case, 
B\'enyi-Torres 
\cite[Proposition 1]{BT-2004} proved that 
if $1\le p, q, r<\infty$ satisfy $1/p + 1/q = 1/r$ 
then there exists a multiplier $\sigma$ satisfying the condition 
\eqref{S00} for which the corresponding 
operator $T_{\sigma}$ is not bounded 
from $L^p \times L^q$ to $L^r$.

The subject of the present paper concerns with 
bilinear multipliers satisfying the condition \eqref{S00}. 
We shall consider the boundedness of $T_{\sigma}$ 
only on $L^2 \times L^2$. 
The above result of 
B\'enyi-Torres 
\cite{BT-2004} implies 
that we can expect the 
boundedness of $T_{\sigma}$ 
only if we strengthen the condition \eqref{S00}. 
So far several such results are known. 
As far as the boundedness 
of operators $T_{\sigma}$ on $L^2 \times L^2$ 
is concerned, it seems that the most 
general result known so far is the theorem given in 
Kato-Miyachi-Tomita \cite{KMT-arXiv}, which reads as follows. 
For nonnegative function $W$ on $\Z^n \times \Z^n$, 
the class $BS_{0,0}^{W}(\R^n)$ is 
defined to be all 
$C^{\infty}$ functions $\sigma$ on $\R^n \times \R^n$ 
satisfying the estimate 
\begin{equation*}
|\partial_{\xi}^{\alpha} \partial_{\eta}^{\beta} \sigma (\xi, \eta)|
\le C_{\alpha, \beta} 
\sum_{\mu, \nu \in \Z^n} 
W(\mu, \nu) \ichi_{R}(\xi - \mu) \ichi_{R} (\eta - \nu), 
\end{equation*}
where 
$R=[-1/2, 1/2)^{n}$ 
(this notation is slightly different from the one given in \cite{KMT-arXiv}). 
The theorem of \cite[Theorem 1.3]{KMT-arXiv} asserts that 
$T_{\sigma}: L^2 \times L^2 \to L^r$, $1\le r\le 2$,  
for all 
$\sigma\in BS_{0,0}^{W}(\R^n)$ 
if and only if 
\begin{equation}\label{calBB}
\|W\|_{\calB}
=
\sup 
\sum_{\mu, \nu \in \Z^n} W(\mu, \nu) F(\mu) G(\nu) H(\mu+\nu) 
< \infty, 
\end{equation}
where the $\sup$ is taken over all 
nonnegative sequences $F, G, H \in \ell^{2}(\Z^n)$ with 
$\|F\|_{\ell^2}=\|G\|_{\ell^2}=\|H\|_{\ell^2}=1$. 
It is known that all nonnegative 
$W$ in the Lorentz class $\ell^{4, \infty}(\Z^{2n})$  
satisfies the condition \eqref{calBB}; see \cite[Proposition 3.4]{KMT-arXiv}. 
In particular $W(\mu, \nu) = (1+|\mu|+|\nu|)^{-n/2}$ is a typical 
example of $W$ satisfying \eqref{calBB}. 
The $L^2\times L^2 \to L^1$ estimates for 
operators in the class 
$BS_{0,0}^{W}(\R^n)$ with 
$W(\mu, \nu) = (1+|\mu|+|\nu|)^{-n/2}$ is also given 
in \cite{MT-2013},  
where estimates in other function spaces are given as well. 
Slav\'{i}kov\'{a} \cite{Slavikova} also 
gives the $L^2\times L^2 \to L^1$ estimates 
of $T_{\sigma}$ for $\sigma\in BS_{0,0}^{W}(\R^n)$ 
with $W \in \ell^{4, \infty}(\Z^{2n})$  
even under restricted smoothness assumptions. 
Related results are also given in Grafakos-He-Slav\'{i}kov\'{a} \cite{GHS}. 


Now the theorem  
\cite[Theorem 1.3]{KMT-arXiv} 
is sharp in itself 
but it is a theorem that is concerned with 
a class of multipliers not with an individual multiplier.  
In the present paper, we shall give a theorem 
that treats individual multiplier, 
although we strongly restrict the form of multipliers. 
Here we recall the work of 
Grafakos and Kalton \cite{GK-Marcinkiewicz} that
considers a problem in the same spirit.

In \cite{GK-Marcinkiewicz}, the authors consider 
the multiplier $\sigma$ of the form  
\[
\sigma_{A} (\xi, \eta) 
=
\sum_{j, k \in \Z} a_{j,k} \phi (2^{-j} \xi) \phi (2^{-k}\eta), 
\quad \xi, \eta  \in \R^n, 
\]
where 
$A=(a_{j,k})_{j,k\in Z}$ is a infinite matrix of complex numbers 
and $\phi$ is a function in $C_{0}^{\infty}(\R^n)$ such 
that $\supp \phi \subset \{2^{-1} \le |\xi|\le 2\}$  
and 
$\sum_{j \in \Z} \phi (2^{-j} \xi) = 1$ for all $\xi \in \R^n 
\setminus \{0\}$. 
The authors of \cite{GK-Marcinkiewicz} introduce a norm 
$H(A)$ for infinite matrices $A$ 
and prove  
the inequality 
\begin{equation}\label{GKHA}
c^{-1} H(A) \le 
\|T_{\sigma_A}\|_{H^p \times H^q \to L^r} 
\le 
c H(A), 
\end{equation}
where 
$H^p$, $0<p<\infty$, 
denotes Hardy spaces on $\R^n$ 
and $p,q,r$ are 
positive real numbers 
satisfying $1/p + 1/q = 1/r$ 
(see \cite[Theorem 6.5]{GK-Marcinkiewicz}).  
Here we do not give the definition of 
$H(A)$ but we mention that it depends only on 
$A$ but not on $p,q,r$. 
As an application of this theorem, 
the authors give an example of a multiplier $\sigma (\xi, \eta)$ 
that satisfies the Marcinkiewicz type condition 
\[
|
\partial_{\xi}^{\alpha}
\partial_{\eta}^{\beta}
\sigma (\xi, \eta)| 
\le C_{\alpha} |\xi|^{-|\alpha|} |\eta|^{- |\beta|}
\]
but the corresponding operator 
$T_{\sigma}$ is not bounded 
in $H^p \times H^q \to L^r$ for 
any $0<p,q,r<\infty$ with $1/p + 1/q = 1/r$.

In the present paper, 
we shall give a theorem that is similar to \eqref{GKHA} 
in relation to multipliers satisfying the condition 
\eqref{S00}. 
Although our result is not directly related to 
the result of \cite{GK-Marcinkiewicz}, 
we shall use some ideas given in this paper.

A study of similar nature can be found in  
Buri\'ankov\'a-Grafakos-He-Honz\'ik 
\cite{BGHH}.

\subsection{Main results}
Now we shall give the statement of the 
main result of this paper.

If $a_{\mu, \nu}\in \C$ is given 
for each $(\mu, \nu)\in \Z^n \times \Z^n$, 
then we call $A=(a_{\mu, \nu})_{\mu, \nu \in \Z^n}$ 
a matrix on $\Z^n \times \Z^n$. 
If 
$\sup_{\mu, \nu} |a_{\mu, \nu}|<\infty$, 
then we say  
$A$ is an $L^{\infty}$ {\it matrix}. 

If $A=(a_{\mu, \nu})_{\mu, \nu \in \Z^n}$ is an $L^{\infty}$ matrix 
on $\Z^n \times \Z^n$, 
and if 
$\Phi \in C_{0}^{\infty}(\R^n \times \R^n)$, 
we define 
the function $\sigma_{A, \Phi}$ by 
\[
\sigma_{A,\Phi} (\xi, \eta)= \sum_{\mu, \nu \in \Z^n}
a_{\mu, \nu} \Phi (\xi - \mu, \eta - \nu), 
\quad 
\xi, \eta \in \R^n.  
\]
Notice that we always have 
$\sigma_{A, \Phi}\in L^{\infty}(\R^n \times \R^n)$.

The purpose of the present paper is to 
show that under certain condition on $\Phi$ 
we can completely characterize 
all $L^{\infty}$ matrices $A$ that satisfy 
\begin{equation}\label{abc}
T_{\sigma_{A, \Phi}} : L^2 \times L^2 \to 
(L^2, \ell^q), 
\quad 
1\le q \le \infty.  
\end{equation}
where $(L^2, \ell^q)$ is the $L^2$-based amalgam space. 
The precise definition of the amalgam space 
will be given in the next section. 
Our characterization of $A$ for \eqref{abc} 
does not depend on $q\in [1, \infty]$, 
which would be of independent interest.

We shall introduce a norm for $L^{\infty}$ matrices. 
For this, 
we write $\ell^2 = \ell^2 (\Z^n)$ to denote 
the class of all $F: \Z^n \to \C$ such that 
\[
\|F\|_{\ell^2}=\bigg(	\sum_{\mu \in \Z^n} |F(\mu)|^2 \bigg)^{1/2}<\infty.  
\]
We also write $\ell^2_0 = \ell^2_0 (\Z^n)$ to denote 
the class of 
all those $F\in \ell^{2}$ such that $F(\mu)=0$ except 
for a finite number of $\mu \in \Z^n$. 
The norm for $L^{\infty}$ matrices is defined 
as follows.

\begin{defn}\label{calB}
If 
$A=(a_{\mu, \nu})_{\mu, \nu \in \Z^n}$ is an $L^{\infty}$ matrix, 
then 
$\|A\|_{\calB}$ denotes  
the norm of the trilinear functional 
\[
(\ell^{2}_{0}(\Z^n))^3 
\ni (F,G,H) 
\mapsto 
\sum_{\mu, \nu \in \Z^n} 
a_{\mu, \nu}
F (\mu) G(\nu) H(\mu + \nu)
\in \C,  
\]
i.e., 
$\|A\|_{\calB}$ is the sup of 
\[
\bigg|
\sum_{\mu, \nu \in \Z^n} 
a_{\mu, \nu}
F (\mu) G(\nu) H(\mu + \nu)
\bigg|
\]
over all $F, G, H \in \ell^{2}_{0}(\Z^n)$ satisfying 
$\|F\|_{\ell^2}=
\|G\|_{\ell^2}=
\|H\|_{\ell^2}=1$. 
\end{defn}

We also use the following. 
\begin{defn}\label{conditionA}
We say that a function $\Phi \in C_{0}^{\infty}(\R^{d})$ satisfies 
the condition (A) if 
there exists a 
function 
$\Theta \in C_{0}^{\infty}(\R^{d})$ such that 
\[
\int_{\R^{d}} \Theta (x ) \Phi (x - \alpha)\, dx 
=
\begin{cases}
1 & \text{if}\;\; \alpha = 0 \in \Z^{d}, 
\\
0 & \text{if} \;\; \alpha \in \Z^{d}\setminus \{0\}. 
\end{cases} 
\]
\end{defn}

Obviously  
any nonzero function $\Phi \in C_{0}^{\infty}(\R^{d})$ 
with support included in the cube $[-1/2, 1/2]^d$ satisfies the condition (A). 
It may not be so obvious that there exists a nonzero 
$\Phi \in C_{0}^{\infty}(\R^{d})$ that does not satisfy the 
condition (A). 
In Section \ref{examples}, we shall give 
some examples of $\Phi$ that satisfy or do not satisfy 
the condition (A).

The following 
is the main result of this paper. 
\begin{thm}\label{main}
\rm{(1)} 
For every  
$\Phi \in C_{0}^{\infty}(\R^n \times \R^n)$, 
there exists a constant $c\in (0, \infty)$ depending only on 
$n$ and $\Phi$ such that 
\begin{equation}\label{lecalB1}
\|\sigma_{A, \Phi}\|_{\calM (L^2 \times L^2 \to (L^2, \ell^{1}))}
\le 
c \|A\|_{\calB}
\end{equation}
for all $L^{\infty}$ matrices $A$. 
\\
\rm{(2)}  
If $\Phi \in C_{0}^{\infty}(\R^n \times \R^n)$ satisfies the 
condition (A), then 
exists a constant $c\in (0, \infty)$ depending only on 
$n$ and $\Phi$ such that 
\begin{equation}\label{gecalBinfty}
\|\sigma_{A, \Phi}\|_{\calM (L^2 \times L^2 \to (L^2, \ell^{\infty}))} 
\ge c^{-1}\|A\|_{\calB}
\end{equation}
for all $L^{\infty}$ matrices $A$. 
\end{thm}

Since the norms of the amalgam spaces satisfy 
\[
\|f \|_{(L^2, \ell^{\infty})}\le 
\|f\|_{(L^2, \ell^q)}
\le 
\|f\|_{(L^2, \ell^1)}, 
\quad 
1 \le q \le \infty, 
\]
the following is an immediate corollary to the above theorem. 
\begin{cor}\label{COR}
If $\Phi \in C_{0}^{\infty}(\R^n \times \R^n)$ satisfies the 
condition (A), 
then 
there exists a constant $c\in (0, \infty)$ depending only on 
$n$ and $\Phi$ such that 
\[
c^{-1}\|A\|_{\calB}
\le 
\|\sigma_{A, \Phi}\|_{ \calM (L^2 \times L^2 \to (L^2, \ell^q)) }
\le 
c \|A\|_{\calB}
\]
for all $L^{\infty}$ matrices $A$ and all $1\le q \le \infty$. 
\end{cor}

We recall that the $L^2$-based amalgam spaces satisfy 
the following inclusion relations: 
\begin{align*}
&
 (L^2, \ell^2)=L^2, 
\\
&
(L^{2}, \ell^{q}) \hookrightarrow L^{q}   
\;\; \text{if}\;\; 1 < q< 2, 
\\
&
(L^2,\ell^1) \hookrightarrow 
h^1 
 \hookrightarrow 
 L^1, 
\\
&
L^{q} \hookrightarrow (L^{2}, \ell^{q})  
\;\;  \text{if}\;\; 2 < q<  \infty, 
\\
&
L^{\infty} 
 \hookrightarrow 
bmo
 \hookrightarrow 
(L^2,\ell^{\infty}), 
\end{align*}
where $h^1$ is the local Hardy space 
and $bmo$ is the local BMO space given by 
Goldberg \cite{goldberg 1979} 
(a proof of the embedding 
$(L^2,\ell^1) \hookrightarrow 
h^1$ can be found in \cite[\S 2.3]{KMT-JMSJ}).  
Thus Corollary \ref{COR} 
implies nontrivial results for 
mapping properties of 
$T_{\sigma_{A, \Phi}}$ in Lebesgue spaces, $h^1$, or $bmo$.  
For example, it implies that 
if 
$T_{\sigma_{A, \Phi}}: L^2 \times L^2 \to L^2$ 
then 
$T_{\sigma_{A, \Phi}}: L^2 \times L^2 \to h^1\hookrightarrow L^1$.

\section{Amalgam spaces}
Here we recall the definition of the amalgam spaces. 
We use the following 
notation to denote 
the cubes in $\R^n$: 
\begin{align*}
&
Q= [-1/2, 1/2]^{n}, 
\\
&
aQ = [-a/2, a/2]^n, 
\quad a \in (0,\infty). 
\end{align*}
If $X,Y,Z$ are function spaces equipped with norms, then 
we write the mixed norm as 
\begin{equation*}
\label{normXYZ}
\| f (x,y,z) \|_{ X_x Y_y Z_z } 
= 
\| \| \| f (x,y,z) 
\|_{ X_x } \|_{ Y_y } \|_{ Z_z }
\end{equation*} 
(pay attention to the order 
of taking norms).     
We shall use these mixed norms for  
$X, Y, Z$ being $L^p$ or $\ell^p$.

The definition of the amalgam space reads as follows.

\begin{defn}
For measurable functions $f$ on 
$\R^n$ and for 
$0<p, q \leq \infty$, 
we define 
\begin{equation*}
\| f \|_{ (L^p,\ell^q)} 
=
 \| f(x+\nu) \|_{L^p_x (Q) \ell^q_{\nu}(\Z^n) }
=\left\{ \sum_{\nu \in \Z^n}
\left( \int_{Q} \big| f(x+\nu) \big|^p \, dx 
\right)^{q/p} \right\}^{1/q} 
\end{equation*}
with the usual modification when $p$ or $q$ is equal to $\infty$. 
The amalgam space 
$ (L^p,\ell^q)$ is 
defined to be the set of all 
those $f$ satisfying $\| f \|_{ (L^p,\ell^q)} <\infty$. 
\end{defn}

It is obvious that 
$(L^p,\ell^p)= L^p$. 
We have 
\[
(L^{p_1}, \ell^{q_1})
\hookrightarrow 
(L^{p_2}, \ell^{q_2})
\;\;  \text{if} 
\; \;  
p_1 \ge p_2 \;\;\text{and}\;\; 
q_1 \le q_2. 
\]
For $1\leq p,q < \infty$, 
the duality 
\[
(L^p,\ell^q)^\ast =  (L^{p^{\prime}},\ell^{q^{\prime}}), 
\quad 
1/p + 1/p^{\prime}=1, 
\quad 
1/q + 1/q^{\prime}=1, 
\]
holds.

In the present paper, 
we use only the spaces 
$(L^p,\ell^q)$ with $p=2$ and $1\le q \le \infty$. 
For details of amalgam spaces, 
see Fournier--Stewart \cite{fournier stewart 1985} 
or Holland \cite{holland 1975}.

\section{Proof of Theorem \ref{main} (1)}
\label{sectionlecalB}
In this section, 
we prove 
the inequality 
\eqref{lecalB1}. 
%
We shall follow the argument given in 
Kato-Miyachi-Tomita \cite{KMT-arXiv}. 
Some ideas go back to 
Boulkhemair \cite{boulkhemair 1995}.

The proof will be divided into two steps.

In the first step, we assume that $\Phi$ is written as 
$\Phi (\xi, \eta)=u(\xi) v (\eta)$ 
with 
$u, v \in C_{0}^{\infty}(\R^n)$. 
For this $\Phi$, we shall prove that 
if $T$ is a positive number 
satisfying $\supp u \subset TQ$ and $\supp v \subset TQ$ 
then 
\eqref{lecalB1} holds with 
$c=c_{n, T} 
\|u\|_{L^{\infty}} \|v\|_{L^{\infty}}$. 
We assume 
$A$ is an arbitrary $L^{\infty}$ matrix on $\Z^n \times \Z^n$. 
By duality, 
it is sufficient to show 
the inequality 
\begin{equation}\label{bbb}
\bigg|
\int_{\R^n} 
T_{\sigma_{A,\Phi}} (f, g)(x) h(x)\, dx
\bigg|
\le c_{n, T} \|u\|_{L^{\infty}} \|v\|_{L^{\infty}} 
\|A\|_{\calB}
\|{f}\|_{L^2 }
\|{g}\|_{L^2}
\|h\|_{(L^2, \ell^{\infty})}
\end{equation}
for all $f, g \in \calS $ and 
all 
$h \in (L^2, \ell^{\infty})$ with compact support.

The integral in 
the left hand side of 
\eqref{bbb} can be written as 
\begin{align*}
&
\int_{\R^n} 
T_{\sigma_{A,\Phi}}(f,g)(x) h(x)\, dx
\\
&
=
\sum_{\mu, \nu \in \Z^n} 
\iiint_{\xi, \eta, x  \in \R^n}
e^{2\pi i x \cdot (\xi + \eta)} 
a_{\mu, \nu} u (\xi -\mu) v (\eta - \nu) 
\widehat{f}(\xi)
\widehat{g}(\eta) 
h(x)\, d\xi d\eta dx
\\
&
=
\sum_{\mu, \nu, \rho \in \Z^n} 
\iiint_{\substack{
\xi, \eta \in TQ 
\\ x \in Q
}}
e^{2\pi i (x+\rho) \cdot (\xi + \mu + \eta + \nu)} 
a_{\mu, \nu} 
u (\xi) v (\eta ) 
\widehat{f}(\xi + \mu)
\widehat{g}(\eta + \nu) 
h(x+ \rho)\, d\xi d\eta dx
\\
&=(\ast). 
\end{align*}
We write 
\begin{align*}
&
e^{2\pi i (x+\rho) \cdot (\xi + \mu + \eta + \nu)} 
=
e^{2\pi i x \cdot (\xi + \eta)}
e^{2\pi i x \cdot (\mu + \nu)}
e^{2\pi i \rho \cdot \xi }
e^{2\pi i \rho \cdot \eta}  
\\
&=
e^{2\pi i x \cdot (\mu + \nu)}
e^{2\pi i \rho \cdot \xi }
e^{2\pi i \rho \cdot \eta}
\sum_{\alpha\in (\Nzero)^n} 
\frac{1}{\alpha !} (2\pi i)^{|\alpha|} 
x^{\alpha}\xi^{\alpha}
\sum_{\beta\in (\Nzero)^n} 
\frac{1}{\beta !} (2\pi i)^{|\beta|} 
x^{\beta}\eta^{\beta}
\end{align*}
(notice that 
$e^{2\pi i \rho \cdot (\mu + \nu)}=1$ 
since 
$\rho \cdot (\mu + \nu)$ is an integer). 
Then we have 
\begin{align*}
(\ast)
&=
\sum_{\mu, \nu , \rho \in \Z^n} 
\sum_{\alpha, \beta\in (\Nzero)^n} 
\frac{(2\pi i)^{|\alpha|} }{\alpha !} \frac{(2\pi i)^{|\beta|} } {\beta !}  
\\
&\qquad 
\times a_{\mu, \nu}
\bigg( \int_{TQ}
e^{2\pi i \rho \cdot \xi} 
u (\xi) \xi^{\alpha} 
\widehat{f}(\xi+ \mu)\, d\xi\bigg)
\bigg( \int_{TQ} 
e^{2\pi i \rho \cdot \eta} 
v (\eta) \eta^{\beta} 
\widehat{g}(\eta+\nu)\, d\eta\bigg)
\\
&\qquad 
\times 
\bigg(\int_{Q}e^{2\pi i  x \cdot (\mu + \nu)}
x^{\alpha + \beta}h(x+ \rho)\, dx
\bigg).  
\end{align*}
In order to estimate this, 
we take 
the sums $\sum_{\mu, \nu}$, 
$\sum_{\rho}$, and 
$\sum_{\alpha, \beta}$ in this order. 
First, we estimate the 
sum $\sum_{\mu, \nu}$ by using 
$\|A\|_{\calB}$ and 
Parseval's identity to obtain 
\begin{align*}
&\bigg|
\sum_{\mu, \nu \in \Z^n} 
a_{\mu, \nu} 
\bigg( \int_{TQ}
e^{2\pi i \rho \cdot \xi} 
u (\xi) \xi^{\alpha} \widehat{f}(\xi + \mu)\, d\xi\bigg)
\bigg( \int_{TQ} 
e^{2\pi i \rho \cdot \eta} 
v (\eta) \eta^{\beta} \widehat{g}(\eta + \nu)\, d\eta\bigg)
\\
&\qquad
\times \bigg(\int_{Q}e^{2\pi i  x \cdot (\mu + \nu)}
x^{\alpha + \beta} h(x+ \rho)\, dx
\bigg) 
\bigg|
\\
&
\le 
\|A\|_{\calB}
\bigg\| \int_{TQ}e^{2\pi i \rho \cdot \xi} 
u (\xi) \xi^{\alpha} 
\widehat{f}(\xi + \mu)\, d\xi\bigg\|_{\ell^2_{\mu}}
\bigg\| \int_{TQ}
e^{2\pi i \rho \cdot \eta} 
 v (\eta) \eta^{\beta} 
\widehat{g}(\eta + \nu)\, d\eta\bigg\|_{\ell^2_{\nu}}
\\
&\qquad 
\times \bigg\|\int_{Q}e^{2\pi i  x \cdot\tau  }
x^{\alpha + \beta} h(x+ \rho)\, dx
\bigg\|_{\ell^2_{\tau}}
\\
&
=
\|A\|_{\calB}
\bigg\| \int_{TQ}e^{2\pi i \rho \cdot \xi} 
u (\xi) \xi^{\alpha}
 \widehat{f}(\xi + \mu)\, d\xi\bigg\|_{\ell^2_{\mu}}
\bigg\| \int_{TQ}
e^{2\pi i \rho \cdot \eta} 
v (\eta) \eta^{\beta} 
\widehat{g}(\eta + \nu)\, d\eta\bigg\|_{\ell^2_{\nu}}
\\
&\qquad 
\times
\|x^{\alpha + \beta} h(x+ \rho)\|_{L^2_{x} (Q)}
\\
&
=(\ast\ast\ast). 
\end{align*}
Next, 
we estimate the sum $\sum_{\rho}$ by using 
the Cauchy-Schwarz inequality,  
Parseval's identity, and Plancherel's theorem  
to obtain 
\begin{align*}
&\sum_{\rho \in \Z^n}
(\ast\ast\ast)
\\
&
\le 
\|A\|_{\calB}
\bigg\| \int_{TQ}e^{2\pi i \rho \cdot \xi} 
u (\xi) \xi^{\alpha}
 \widehat{f}(\xi + \mu)\, d\xi
 \bigg\|_{\ell^2_{\mu}\ell^2_{\rho}}
\bigg\| \int_{TQ}e^{2\pi i \rho \cdot \eta} 
v (\eta) \eta^{\beta} 
\widehat{g}(\eta + \nu)\, d\xi
\bigg\|_{\ell^2_{\nu}\ell^2_{\rho}}
\\
&\qquad 
\times \sup_{\rho}
\|x^{\alpha + \beta} h(x+ \rho)\|_{L^2_{x} (Q)}
\\
&
\le c_{n,T}
\|A\|_{\calB}
\|
u (\xi) \xi^{\alpha} \widehat{f}(\xi + \mu) 
\|_{L^2_{\xi}(TQ)\ell^2_{\mu}}
\| 
v(\eta) \eta^{\beta} 
\widehat{g}(\eta + \nu)
\|_{L^2_{\eta}(TQ) \ell^2_{\nu}}
\\
&\qquad 
\times \sup_{\rho}
\|x^{\alpha + \beta} h(x+ \rho)\|_{L^2_{x} (Q)}
\\
&
\le c_{n,T} T^{|\alpha|+ |\beta|}
\|A\|_{\calB}
\|u\|_{L^{\infty}} 
\|v\|_{L^{\infty}} 
 \|\widehat{f}(\xi + \mu) 
\|_{L^2_{\xi}(TQ)\ell^2_{\mu}}
\| 
\widehat{g}(\eta + \nu)
\|_{L^2_{\eta}(TQ) \ell^2_{\nu}}
\sup_{\rho}
\|h(x+ \rho)\|_{L^2_{x} (Q)}
\\
&
\le c_{n,T} T^{|\alpha|+ |\beta|}
\|A\|_{\calB}
\|u\|_{L^{\infty}} 
\|v\|_{L^{\infty}} 
\|\widehat{f}\|_{L^2 (\R^n)}
\|\widehat{g}\|_{L^2 (\R^n)}
\sup_{\rho}
\|h(x+ \rho)\|_{L^2_{x} (Q)}
\\
& 
= 
c_{n, T} T^{|\alpha|+ |\beta|}
\|A\|_{\calB}
\|u\|_{L^{\infty}} 
\|v\|_{L^{\infty}} 
\|{f}\|_{L^2(\R^n) }
\|{g}\|_{L^2(\R^n)}
\|h\|_{(L^2, \ell^{\infty})(\R^n)} 
\end{align*}
(the constant $c_{n,T}$ in different places are not the same). 
Finally, we estimate 
the sum 
$\sum_{\alpha, \beta}$ by using 
\[
\sum_{\alpha, \beta\in (\Nzero)^n} 
\frac{(2\pi T)^{|\alpha|} }{\alpha !} \frac{(2\pi T)^{|\beta|} } {\beta !} 
=c_{n,T} < \infty
\]
and obtain \eqref{bbb}.

In the second step, we prove that the inequality \eqref{lecalB1} 
holds for every 
$\Phi \in C_{0}^{\infty}(\R^n \times \R^n)$. 
We shall use the idea of using Fourier expansion 
(this idea may trace back to \cite{CM}). 
Take a positive number $T$ satisfying 
$\supp \Phi \subset 2^{-1}TQ\times 2^{-1}TQ$ and take 
a function 
$\phi\in C_{0}^{\infty}(\R^n)$ such that 
$\supp \phi \subset TQ$ and 
$\phi (\xi)=1$ on $2^{-1}TQ$. 
Since $\Phi $ is a smooth function supported in 
the interior of 
$TQ \times TQ$, 
the Fourier series expansion 
gives 
\[
\Phi (\xi, \eta)
=
\sum_{k, \ell \in \Z^n} b_{k,\ell} 
e^{2\pi i k \cdot \xi/T}
e^{2\pi i \ell \cdot \eta/T}, 
\quad (\xi, \eta)\in TQ \times TQ,  
\]
where $b_{k,\ell}$ is a rapidly decreasing sequence. 
From the choice of the function $\phi$, we have 
\begin{equation}\label{FourierseriesExp}
\Phi (\xi, \eta)
=
\sum_{k, \ell \in \Z^n} b_{k,\ell} 
e^{2\pi i k \cdot \xi/T}
e^{2\pi i \ell \cdot \eta/T}
\phi (\xi)\phi (\eta), 
\quad (\xi, \eta)\in \R^n \times \R^n. 
\end{equation}
Now the estimate proved in the first step yields 
\begin{equation*}
\bigg\|
\sum_{\mu, \nu \in \Z^n} 
a_{\mu, \nu} 
e^{2\pi i k \cdot (\xi- \mu)/T}
\phi (\xi - \mu)
e^{2\pi i \ell \cdot (\eta- \nu)/T}
\phi (\eta - \mu)
\bigg\|_{\calM (L^2 \times L^2 \to X)} 
\le c_{n,T} 
\|A\|_{\calB}
\end{equation*}
for each $(k, \ell)$ with the constant 
$c_{n,T}$ independent of $k, \ell$. 
Thus since 
$b_{k,\ell}$ is rapidly decreasing 
we obtain 
\eqref{lecalB1}. 
This proves the part (1) of Theorem \ref{main}. 
\section{
Proof of Theorem \ref{main} (2)} 
\label{sectiongecalB}

In this section, 
we prove 
the inequality 
\eqref{gecalBinfty}. 
We shall first consider $\Phi$ of a special form and then 
consider general $\Phi$.

\subsection{The case of a special $\Phi (\xi, \eta)$}
\label{subsectionproduct}
In this subsection 
we shall prove the 
inequality \eqref{gecalBinfty} for the function 
$\Phi$ defined by 
$\Phi (\xi, \eta)=\phi (\xi) \phi (\eta)$ with 
$\phi$ satisfying  
\[
\phi \in C_{0}^{\infty}(\R^n), 
\quad 
\supp \phi \subset Q, 
\quad 
\phi (\xi) = 1 \;\; \text{for}\; \; 
\xi \in 2^{-1}Q.  
\]
We assume 
$A$ is an arbitrary $L^{\infty}$ matrix, 
and assume 
$F, G, H\in \ell^{2}_{0}(\Z^n)$ satisfy 
$\|F\|_{\ell^2}=
\|G\|_{\ell^2}=
\|H\|_{\ell^2}=1$. 
We shall prove 
that there exist $f, g \in \calS$ such that  
\begin{equation}\label{xxx}
\|f\|_{L^2}\approx \|g\|_{L^2}\approx 1
\end{equation}
and 
\begin{equation}\label{aaa}
\|T_{\sigma_{A,\Phi}}(f, g)\|_{(L^2, \ell^{\infty})}
\gtrsim 
\bigg|
\sum_{\mu, \nu \in \Z^n} 
a_{\mu, \nu}
F(\mu) G(\nu) H(\mu + \nu)
\bigg|, 
\end{equation}
where 
the constant 
in the above $\approx$ and $\gtrsim$ depend 
only on 
$n$ and $\phi$. 
This certainly implies the desired estimate 
\eqref{gecalBinfty}.

We take a function $\theta \in C_{0}^{\infty}(\R^n)$ such that 
$\supp \theta \subset 2^{-1}Q$ and 
$|\calF^{-1}\theta (x)| \gtrsim 1$ for $x \in Q$, 
and define $f, g$ by 
\begin{equation*}
\widehat{f}(\xi) = \sum_{\mu \in \Z^n} F(\mu) 
\theta (\xi - \mu), 
\quad 
\widehat{g}(\eta) = \sum_{\nu \in \Z^n} G(\nu) 
\theta (\eta - \nu). 
\end{equation*}
Obviously 
$f, g\in \calS (\R^n)$. 
We shall prove that 
these 
$f$ and $g$ satisfy 
\eqref{xxx} and 
\eqref{aaa}.

The estimate \eqref{xxx} is obvious since 
\[
\|f\|_{L^2}=\|\widehat{f}\|_{L^2} \approx \|F\|_{\ell^2}=1, 
\quad 
\|g\|_{L^2}=\|\widehat{g}\|_{L^2} \approx \|G\|_{\ell^2}=1. 
\]
To prove \eqref{aaa}, 
notice that our choice of $\phi$ and $\theta$ implies 
\[
\phi (\xi-\mu) \widehat{f}(\xi) = F(\mu)\theta (\xi-\mu), 
\quad 
\phi (\eta-\nu) \widehat{g}(\xi) = G(\nu)\theta (\eta-\nu) 
\]
for all $\mu, \nu \in \Z^n$. 
Hence 
\begin{align*}
&
T_{\sigma_{A,\Phi}}(f, g)(x) 
\\
&
=
\sum_{\mu, \nu \in \Z^n} 
a_{\mu, \nu}
\iint 
e^{2\pi i x \cdot (\xi + \eta)} 
\phi (\xi-\mu) 
\phi (\eta-\nu) 
\widehat{f}(\xi) 
\widehat{g}(\eta) 
\, d \xi d\eta
\\
&
=
\sum_{\mu, \nu \in \Z^n} 
a_{\mu, \nu}
F(\mu) G(\nu) 
\iint e^{2\pi i x \cdot (\xi + \eta )} 
\theta (\xi-\mu ) 
\theta (\eta - \nu) 
\, d \xi d\eta
\\
&
=
\sum_{\mu, \nu \in \Z^n} 
a_{\mu, \nu}
F(\mu) G(\nu) 
\iint e^{2\pi i x \cdot (\xi + \mu + \eta + \nu )} 
\theta (\xi ) 
\theta (\eta ) 
\, d \xi d\eta
\\
&
=
\sum_{\mu, \nu \in \Z^n} 
a_{\mu, \nu}
F (\mu) G(\nu) 
e^{2 \pi i x \cdot (\mu + \nu)}
(\calF^{-1}{\theta} (x))^2 . 
\end{align*}
Let 
$h\in L^2 (Q)$ be the function defined by 
\[
(\calF^{-1}{\theta} (x))^2 h(x) = 
\sum_{\rho \in \Z^n} H(\rho) e^{- 2 \pi i \rho \cdot x}, 
\quad 
x \in Q. 
\]
Then,  
since $|\calF^{-1}{\theta} (x)|\gtrsim 1$ on $Q$, 
Parseval's identity implies  
$
\|h\|_{L^2(Q)}\approx \|H\|_{\ell^2}=1$. 
We have 
\begin{align*}
&
\int_{Q}T_{\sigma_{A,\Phi}}(f, g)(x) h(x)\, dx 
\\
&
=
\sum_{\mu, \nu \in \Z^n} 
a_{\mu, \nu}
F (\mu) G(\nu) 
\int_{Q}
e^{2 \pi i x \cdot (\mu + \nu)}
(\calF^{-1}{\theta} (x))^2 
h(x)\, dx
\\
&
=
\sum_{\mu, \nu \in \Z^n} 
a_{\mu, \nu}
F (\mu) G(\nu) H(\mu + \nu). 
\end{align*}
Thus 
\begin{align*}
&
\| T_{\sigma_{A,\Phi}}(f, g) \|_{(L^2, \ell^{\infty})}
\ge 
\| T_{\sigma_{A,\Phi}}(f, g) \|_{L^2(Q)}
\gtrsim 
\bigg|
\int_{Q}
T_{\sigma_{A,\Phi}}(f, g)(x) h(x)\, dx
\bigg|
\\
&
=
\bigg|
\sum_{\mu, \nu \in \Z^n} 
a_{\mu, \nu}
F (\mu) G(\nu) H(\mu + \nu)
\bigg|, 
\end{align*}
as desired. 
Thus the estimate 
\eqref{gecalBinfty} is proved for the special $\Phi$.

\subsection{The case of general $\Phi (\xi, \eta)$}
\label{general}
The argument in this subsection follows the ideas given by 
Grafakos and Kalton \cite[Proposition 6.2 and Lemma 6.3]{GK-Marcinkiewicz}. 

In order to simplify notation, we write 
$X=(L^2, \ell^{\infty})$. 
We also use the following notation:  
for $m \in L^{\infty}(\R^n)$, 
we write 
$\|m \|_{\calM (L^2 \to L^2)}$ to denote 
the $L^2 \to L^2$ operator norm of the 
linear Fourier multiplier operator 
$f \mapsto \calF^{-1}(m\widehat{f})$. 
By Plancherel's theorem, 
we have in fact 
$\|m \|_{\calM (L^2 \to L^2)}= \|m\|_{L^{\infty}}$.

We shall give the argument in a sequence of lemmas. 

\begin{lem}\label{translation}
The equality 
\[
\|\sigma (\xi + \xi_0, \eta + \eta_{0})\|_{\calM (L^2 \times L^2 \to X)}
=
\|\sigma (\xi , \eta )\|_{\calM (L^2 \times L^2 \to X)} 
\]
holds for all $\sigma \in L^{\infty}(\R^n \times \R^n)$ and 
all $\xi_{0}, \eta_{0}\in \R^n$. 
\end{lem}

\begin{proof}
If we define 
$E_{a}(x) = e^{2\pi i a \cdot x }$ for 
$a \in \R^n$ and $x \in \R^n$, 
then we have 
the formula 
\[
T[\sigma (\xi + \xi_0, \eta + \eta_{0})] (f, g)(x)
=
E_{-\xi_0 - \eta_0} (x) 
T[\sigma] 
(E_{\xi_0} f, E_{\eta_0}g ) (x). 
\]
This implies the equality of the lemma since 
multiplication by the 
unimodular function $E_{a}$ does not change 
the norms of 
$L^2$ and 
$X=(L^2, \ell^{\infty})$. 
\end{proof}

\begin{lem}\label{unconditionaltoPhi}
Let 
$\sigma_{\mu, \nu} \in L^{\infty}(\R^n \times \R^n)$ be given for 
each $(\mu, \nu)\in \Z^n \times \Z^n$ and assume 
the following: 
{(i)} $\sup_{\mu, \nu} \|\sigma_{\mu, \nu}\|_{L^{\infty}}<\infty$;  
{(ii)} there exists a $K\in (0, \infty)$ such that 
$\supp \sigma_{\mu, \nu} \subset (\mu, \nu) + KQ\times KQ$ 
for all $\mu, \nu$; 
{(iii)} there exists an $M\in [0, \infty)$ such that 
\begin{equation}\label{unconditionalaaa}
\bigg\|
\sum_{\mu, \nu \in \Z^n} 
\alpha_{\mu}
\alpha^{\prime}_{\nu}
\sigma_{\mu, \nu}
\bigg\|_{\calM (L^2 \times L^2 \to X)} 
\le M 
\end{equation}
for all 
$\alpha_{\mu} \in \{0, 1\}$ and 
$\alpha^{\prime}_{\nu} \in \{0, 1\}$. 
Then for each 
$\Phi \in C_{0}^{\infty}(\R^n\times \R^n)$, 
there exists a constant 
$c$ depending 
only on 
$n, K, \Phi$ such that 
\begin{equation*}
\bigg\|
\sum_{\mu, \nu \in \Z^n} 
\Phi (\xi- \mu, \eta- \nu) 
\sigma_{\mu, \nu}
\bigg\|_{\calM (L^2 \times L^2 \to X)} 
\le c M.  
\end{equation*}
\end{lem}

\begin{proof}
%
We first consider the case where 
$\Phi $ is the product form 
$\Phi (\xi, \eta)=u(\xi) v(\eta)$ with 
$u, v \in C_{0}^{\infty}(\R^n)$. 
In this case, we shall prove that 
if $T$ is a positive number 
satisfying 
$\supp u , \supp v \subset TQ$ then 
there exists a constant 
$c_{n, K, T}$ depending 
only on 
$n, K, T$ such that 
\begin{equation}\label{product}
\bigg\|
\sum_{\mu, \nu \in \Z^n} 
u (\xi- \mu) v(\eta- \nu) 
\sigma_{\mu, \nu}
\bigg\|_{\calM (L^2 \times L^2 \to X)} 
\le c_{n, K, T} \|u\|_{L^{\infty}} \|v\|_{L^{\infty}}M.  
\end{equation}

We take 
a positive integer 
$N$ satisfying 
$N>(T+K)/2$ and we define the equivalence relation 
for 
$\mu, \nu \in \Z^n$ by 
\[
\mu \equiv \nu \; \Leftrightarrow \; 
N^{-1}(\mu - \nu) \in \Z^n. 
\]
Then from the assumptions on the supports 
of $\sigma_{\mu, \nu}$, $u$, and $v$, 
and from our choice of $N$, 
we see that 
\begin{equation*}
\mu\equiv \mu^{\prime}, 
\; 
\nu\equiv \nu^{\prime}, 
\; 
u(\xi - \mu^{\prime})v(\eta - \nu^{\prime})
\sigma_{\mu, \nu} (\xi, \eta) \neq 0, 
\; 
\Rightarrow \; 
\mu = \mu^{\prime}, \; 
\nu = \nu^{\prime}. 
\end{equation*}
Hence, 
for each fixed 
$j, k \in \{0, \dots, N-1\}^{n}$, 
we can write 
\begin{align*}
&
\sum_{\mu \equiv j} 
\sum_{\nu \equiv k}
u(\xi - \mu)v(\eta - \nu)
\sigma_{\mu, \nu} (\xi, \eta) 
\\
&
=
\bigg(\sum_{\mu \equiv j} 
\sum_{\nu \equiv k}
\sigma_{\mu, \nu} (\xi, \eta) 
\bigg)
\bigg( \sum_{\mu^{\prime}\equiv j } 
u(\xi - \mu^{\prime})\bigg)
\bigg(\sum_{\nu^{\prime}\equiv k} 
v(\eta - \nu^{\prime})
\bigg)  
\end{align*}
and hence 
\begin{align*}
&
\bigg\|
\sum_{\mu \equiv j} 
\sum_{\nu \equiv k}
u(\xi - \mu^{\prime})v(\eta - \nu^{\prime})
\sigma_{\mu, \nu} (\xi, \eta)  
\bigg\|_{\calM (L^2 \times L^2 \to X)} 
\\
&
\le 
\bigg\|
\sum_{\mu \equiv j} 
\sum_{\nu \equiv k}
\sigma_{\mu, \nu} 
\bigg\|_{\calM (L^2\times L^2 \to X)}
\bigg\|\sum_{\mu^{\prime}\equiv j} 
u(\xi - \mu^{\prime}) 
\bigg\|_{\calM (L^2 \to L^2)}
\bigg\|\sum_{\nu^{\prime}\equiv k} 
v(\eta - \nu^{\prime}) 
\bigg\|_{\calM (L^2 \to L^2)}. 
\end{align*}
The $\calM (L^2 \times L^2 \to X)$ norm 
on the right hand side does not exceed 
$M$ by the assumption \eqref{unconditionalaaa}. 
For the two 
$\calM (L^2 \to L^2)$ norms,  
Plancherel's theorem 
and the assumptions on the supports of 
$u$ and $v$ yield 
\begin{align*}
&
\bigg\|\sum_{\mu^{\prime}\equiv j} 
u(\xi - \mu^{\prime}) 
\bigg\|_{\calM (L^2 \to L^2)}
=
\bigg\|\sum_{\mu^{\prime}\equiv j } 
u(\xi - \mu^{\prime}) 
\bigg\|_{L^{\infty}}
\le 
c_{n, T} 
\|u\|_{L^{\infty}}, 
\\
&
\bigg\|\sum_{\nu^{\prime}\equiv k} 
v(\eta - \nu^{\prime}) 
\bigg\|_{\calM (L^2 \to L^2)}
=
\bigg\|\sum_{\nu^{\prime}\equiv k} 
v(\eta - \nu^{\prime}) 
\bigg\|_{L^{\infty}}
\le c_{n, T}
\|v\|_{L^{\infty}}.  
\end{align*}
Thus we obtain 
\begin{equation*} 
\bigg\|
\sum_{\mu \equiv j} 
\sum_{\nu \equiv k}
u(\xi - \mu^{\prime})v(\eta - \nu^{\prime})
\sigma_{\mu, \nu} (\xi, \eta)  
\bigg\|_{\calM (L^2 \times L^2 \to X)} 
\le 
c_{n,T}
\|u\|_{L^{\infty}}
\|v\|_{L^{\infty}}
M. 
\end{equation*}
Summing over $(j,k)$, 
we obtain 
\eqref{product}.

The case of general 
$\Phi \in C_{0}^{\infty}(\R^n \times \R^n)$ can be deduced 
from the case of product $\Phi$ by the use of 
the expansion \eqref{FourierseriesExp} 
(see the last part of Section \ref{sectionlecalB}). 
Lemma \ref{unconditionaltoPhi} is proved. 
\end{proof}

\begin{lem}\label{unconditionaltoaverage}
Suppose $\sigma_{\mu, \nu}$, $K$, and $M$ satisfy the 
assumptions (i), (ii), (iii) of Lemma 
\ref{unconditionaltoPhi}. 
Set 
\begin{equation*}
\langle \sigma_{\mu, \nu} \rangle 
=
\iint_{\R^n \times \R^n} 
\sigma_{\mu, \nu} (\xi, \eta)\, 
d\xi d\eta  
\end{equation*}
and define the $L^{\infty}$ matrix $B$ by 
$
B=(\langle \sigma_{\mu, \nu} \rangle )_{\mu, \nu \in \Z^n}$. 
Then 
for every 
$\Phi \in C_{0}^{\infty}(\R^n\times \R^n)$ 
there 
exists a constant 
$c$ depending only on 
$n, K, \Phi$ such that 
\begin{equation*}
\|
\sigma_{B, \Phi}
\|_{\calM (L^2 \times L^2 \to X)} 
\le c M
\end{equation*}
\end{lem}

\begin{proof} 
By Lemma \ref{translation}, we have 
\begin{equation}\label{AAA}
\bigg\|
\sum_{\mu, \nu} \alpha_{\mu} 
\alpha^{\prime}_{\nu} 
\sigma_{\mu, \nu} (\xi + \xi_0, \eta + \eta_0) 
\bigg\|_{\calM (L^2 \times L^2 \to X)}
\le M
\end{equation}
for all 
sequences 
$(\alpha_{\mu})$ and 
$(\alpha^{\prime}_{\nu})$ consisting of $0$ or $1$ 
and for 
all $\xi_0, \eta_{0} \in \R^n$.

Take an arbitrary function 
$\Phi \in C_{0}^{\infty}(\R^n \times \R^n)$ and 
take a $T\in (0, \infty)$ such that 
$\supp \Phi \subset TQ \times TQ$. 
Then the estimate 
\begin{equation}\label{BBB}
\begin{aligned}
&
\bigg\|
\sum_{\mu, \nu} 
\Phi (\xi - \mu, \eta - \nu)
\sigma_{\mu, \nu} (\xi + \xi_0, \eta + \eta_0) 
\bigg\|_{\calM (L^2 \times L^2 \to X)}
\\
&
\le c_{n,K, \Phi} M 
\ichi_{(T+K)Q}(\xi_0)
\ichi_{(T+K)Q}(\eta_0)
\end{aligned}
\end{equation}
holds for all $\xi_0, \eta_0 \in \R^n$. 
In fact, 
if $\xi_0 \not\in (T+K)Q$ or 
$\eta_0 \not\in (T+K)Q$, then 
the function on the left hand side of 
\eqref{BBB} is identically equal to $0$ and the estimate is obvious. 
If $\xi_0 \in (T+K)Q$ and $\eta_0 \in (T+K)Q$, then 
the support of the function 
$\sigma_{\mu, \nu} (\xi+\xi_0, \eta + \eta_0)$ is 
included 
in $(\mu, \nu)+ (T+2K)Q\times (T+2K)Q$ and  
we obtain \eqref{BBB} from 
\eqref{AAA} and Lemma \ref{unconditionaltoPhi}.

Now since the norm in 
$X=(L^2, \ell^{\infty})$ is a Banach-space norm 
(satisfying the triangular inequality), 
we take the 
integral of \eqref{BBB} over 
$\xi_0, \eta_0 \in \R^n$ to 
obtain the desired inequality. 
\end{proof}

\begin{lem}\label{makeuncondtional}
For each $\Phi \in C_{0}^{\infty}(\R^n\times \R^n)$, 
there exists a constant 
$c$ depending only on $n$ and $\Phi$ such that 
the 
inequality 
\begin{equation*}
\bigg\|
\sum_{\mu, \nu \in \Z^n} 
\alpha_{\mu}
\alpha^{\prime}_{\nu}
\Phi (\xi- \mu, \eta- \nu) 
\sigma (\xi, \eta)
\bigg\|_{\calM (L^2 \times L^2 \to X)} 
\le c
\big\|
\sigma
\big\|_{\calM (L^2 \times L^2 \to X)} 
\end{equation*}
holds 
for all 
$\sigma \in L^{\infty}(\R^n \times \R^n)$ and 
all sequences 
$(\alpha_{\mu})$ and 
$(\alpha^{\prime}_{\nu})$ consisting of 
$0$ and $1$. 
\end{lem}

\begin{proof} 
Let $(\alpha_{\mu})$ and 
$(\alpha^{\prime}_{\nu})$ be 
arbitrary sequences consisting of $0$ and $1$. 
If $\Phi$ is of the form 
$\Phi (\xi, \eta)=u(\xi) v(\eta)$, 
then 
we can write 
\begin{equation*}
\sum_{\mu, \nu}
\alpha_{\mu}
\alpha^{\prime}_{\nu}
u(\xi-\mu) v(\eta-\nu ) 
\sigma (\xi, \eta)
=
\bigg(\sum_{\mu}
\alpha_{\mu} 
u(\xi -\mu)\bigg)
\bigg(\sum_{\nu}
\alpha^{\prime}_{\nu} 
v(\xi -\mu)\bigg)
\sigma (\xi, \eta). 
\end{equation*} 
Hence, if $T$ is a positive number 
satisfying 
$\supp u, \supp v \subset TQ$,  
then 
\begin{align*}
&\bigg\|
\sum_{\mu, \nu}
\alpha_{\mu}
\alpha^{\prime}_{\nu}
u(\xi-\mu) v(\eta-\nu ) 
\sigma (\xi, \eta)
\bigg\|_{\calM (L^2 \times L^2 \to X)}
\\
&
\le 
\bigg\|\sum_{\mu}
\alpha_{\mu} u(\xi -\mu)\bigg\|_{\calM (L^2 \to L^2)}
\bigg\|\sum_{\nu}\alpha^{\prime}_{\nu} v(\xi -\mu)\bigg\|_{\calM (L^2 \to L^2)}
\|\sigma (\xi, \eta)\|_{\calM (L^2 \times L^2 \to X)}
\\
&
\le c_{n,T}
\|u\|_{L^{\infty}}
\|v\|_{L^{\infty}}
\|\sigma (\xi, \eta)\|_{\calM (L^2 \times L^2 \to X)}. 
\end{align*}
The case of general $\Phi \in C_{0}^{\infty}(\R^n \times \R^n)$ 
can be reduced to the case of the above $\Phi$ 
by the use of Fourier expansion 
as in the last part of  
Section \ref{sectionlecalB}.  
\end{proof}

\begin{lem}\label{APhitoAtildePhi}
Let 
$\Phi, \widetilde{\Phi} 
\in C_{0}^{\infty}(\R^n \times \R^n)$ and suppose 
$\Phi$ satisfies the condition (A). 
Then there exists a constant $c$ depending only on 
$n, \Phi$, and $\widetilde{\Phi}$ such that 
\begin{equation*}
\|
\sigma_{A, \widetilde{\Phi}}
\|_{\calM (L^2 \times L^2 \to X)} 
\le c 
\|
\sigma_{A, \Phi}
\|_{\calM (L^2 \times L^2 \to X)} 
\end{equation*}
for all 
$L^{\infty}$ matrices $A=(a_{\mu, \nu})$. 
\end{lem}

\begin{proof}
Let 
$A=(a_{\mu, \nu})_{\mu, \nu \in \Z^n}$ be an arbitrary 
$L^{\infty}$ matrix. 
Take a function  
$\Theta\in C_{0}^{\infty}(\R^n \times \R^n)$. 
By Lemma \ref{makeuncondtional} we have 
\begin{equation*}
\bigg\|
\sum_{\mu, \nu} 
\alpha_{\mu} \alpha^{\prime}_{\nu} 
\Theta (\xi - \mu, \eta - \nu)
\sigma_{A, \Phi} (\xi, \eta)
\bigg\|_{\calM (L^2 \times L^2 \to X)}
\le c_{n, \Theta} 
\|\sigma_{A, \Phi}\|_{\calM (L^2 \times L^2 \to X)} 
\end{equation*}
for all sequences $(\alpha_{\mu})$ and 
$(\alpha^{\prime}_{\nu})$ consisting of $0$ and $1$. 
Hence 
Lemma \ref{unconditionaltoaverage} implies 
\begin{equation}\label{tildeA}
\|\sigma_{\widetilde{A}, \widetilde{\Phi}}
\|_{\calM (L^2 \times L^2 \to X)}
\le c_{n, \Theta, \widetilde{\Phi}} 
\|\sigma_{A, \Phi}\|_{\calM (L^2 \times L^2 \to X)} ,   
\end{equation}
where $\widetilde{A}=(\widetilde{a}_{\mu, \nu})$ with 
\begin{align*}
&
\widetilde{a}_{\mu, \nu}=
\iint_{\R^n \times \R^n} 
\Theta (\xi-\mu, \eta-\nu) 
\sigma_{A, \Phi}(\xi, \eta)\, 
d\xi d\eta
\\
&
=\sum_{\mu^{\prime}, \nu^{\prime}}
a_{\mu^{\prime}, \nu^{\prime}} 
\iint_{\R^n \times \R^n} 
\Theta (\xi-\mu, \eta-\nu) 
\Phi (\xi - \mu^{\prime}, \eta - \nu^{\prime})\, 
d\xi d\eta. 
\end{align*}
If we define 
\[
R (\alpha, \beta )
=
\iint_{\R^n \times \R^n} 
\Theta (\xi, \eta) 
\Phi (\xi - \alpha, \eta - \beta)\, 
d\xi d\eta , 
\quad \alpha, \beta \in \Z^n, 
\]
then 
we have 
\[
\widetilde{a}_{\mu, \nu}
=
\sum_{\mu^{\prime}, \nu^{\prime}}
a_{\mu^{\prime}, \nu^{\prime}} 
R(\mu^{\prime}- \mu, \nu^{\prime}- \nu).  
\]
Since $\Phi$ satisfies the condition (A), 
we can choose the 
function $\Theta \in C_{0}^{\infty}(\R^n \times \R^n)$ so that we have 
\[
R(\alpha, \beta) = 
\begin{cases}
1 & \text{if}\;\; (\alpha, \beta)=(0,0)
\\
0 & \text{if}\;\; (\alpha, \beta)\neq (0,0). 
\end{cases}
\]
With this choice of $\Theta$, we have 
$\widetilde{a}_{\mu, \nu}=a_{\mu, \nu}$ and 
\eqref{tildeA} is the desired inequality. 
\end{proof}
\vs

Using Lemma \ref{APhitoAtildePhi}, we can complete the 
proof of 
Theorem \ref{main} (2). 
If $\Phi$ satisfies the condition (A), 
then by Lemma \ref{APhitoAtildePhi} the inequality 

\[
\|
\sigma_{A, \widetilde{\Phi}}
\|_{\calM (L^2 \times L^2 \to X)}
\le c_{n, \Phi, \widetilde{\Phi}} 
\|
\sigma_{A, \Phi}
\|_{\calM (L^2 \times L^2 \to X)}
\]
holds 
for any $\widetilde{\Phi}\in C_{0}^{\infty}(\R^n \times \R^n)$. 
In particular 
the above inequality 
holds for the special function 
$\widetilde{\Phi}(\xi, \eta)= \phi (\xi)\phi(\eta)$ that 
was treated in 
Subsection \ref{subsectionproduct}. 
For this $\widetilde{\Phi}$, 
we have proved 
$\|
\sigma_{A, \widetilde{\Phi}}
\|_{\calM (L^2 \times L^2 \to X)} \ge c^{-1} \|A\|_{\calB}$, 
which combined with the above inequality implies 
the same lower bound 
for 
$\|\sigma_{A, \Phi}\|_{\calM (L^2 \times L^2 \to X)}$. 
Thus Theorem \ref{main} (2) is proved and the proof of Theorem \ref{main} 
is complete.

\section{Remarks on the condition (A)}\label{examples}
In this section, we give 
examples of $\Phi \in C_{0}^{\infty}(\R^{d})$ 
that satisfy or do not satisfy the condition (A).

First, we give an example of $\Phi \in C_{0}^{\infty}(\R^{d})$ 
that satisfies the condition (A).  
This example is essentially the same as the function 
considered in \cite[Lemma 6.3]{GK-Marcinkiewicz} 
in a slightly different situation. 

First consider the case $d=1$. 
Let $\phi$ be a function on $\R$ such that
\begin{equation}\label{phi}
\phi \in C_{0}^{\infty}(\R), 
\quad 
\supp \phi \subset [-1, 1],  
\quad 
\phi (0) \neq 0. 
\end{equation}
This $\phi$ satisfies the condition (A). 
In fact, 
it is easy to see that 
the three functions 
$\phi (x-j)$,  $j=-1, 0, 1$, 
are 
linearly independent on the open interval $(-1,1)$. 
This means that the linear functionals 
\[
C_{0}^{\infty}((-1,1)) \ni \theta  
\mapsto 
\int_{-1}^{1} \theta (x) \phi (x - j ) \, dx \in \C, 
\quad 
j=-1, 0, 1, 
\] 
are linearly independent and hence, 
by an elementary fact of linear algebra,  
there exists a function $\theta \in C_{0}^{\infty}((-1,1))$ 
such that 
\begin{equation*}
\int_{-1}^{1} \theta (x) \phi (x - j)\, dx 
=
\begin{cases}
1 & \text{for}\;\; j =0
\\
0 & \text{for}\;\; j=-1, 1. 
\end{cases}
\end{equation*}
But the integral on the left hand side is 
also equal to $0$ for all $j\in \Z \setminus \{-1, 0, 1\}$ 
since $\phi (\cdot - j)=0$ on $(-1,1)$ for those $j$.  
Thus 
$\theta$ has the property required for the condition (A).

For $d\ge 2$, it is easy to see that 
the function $\Phi$ defined by 
\[
\Phi (x)=\phi (x_1) \cdots \phi (x_d), 
\quad x=(x_1, \dots, x_d) \in \R^d, 
\]
with the function $\phi$ of \eqref{phi}  
satisfies the condition (A).

Next we give an example of nonzero function 
that does not satisfy the condition (A).

We first consider the case $d=1$. 
Let $\phi$ be the function of \eqref{phi} and 
set 
\[
\widetilde{\phi} (x) = \phi (x) + \phi (x-1). 
\]
Then 
\[
\sum_{k\in \Z,\; k\, \text{even}}
\widetilde{\phi} (x-k)
=
\sum_{k\in \Z,\; k\, \text{odd}}
\widetilde{\phi} (x-k)
\]
and hence 
\[
\sum_{k\in \Z} (-1)^{k}
\int_{\R} 
\theta (x) \widetilde{\phi} (x-k)\, dx=0 
\]
for all $\theta \in C_{0}^{\infty}(\R)$. 
Hence  
$\widetilde{\phi}$ does not 
satisfy the condition (A).

For $d\ge 2$, 
the function $\widetilde{\Phi}$ defined by 
\[
\widetilde{\Phi} (x)=\widetilde{\phi} (x_1) \cdots \widetilde{\phi} (x_d), 
\quad x=(x_1, \dots, x_d) \in \R^d, 
\]
satisfies 
\begin{equation*}
\sum_{\alpha \in \Z^d} 
(-1)^{\alpha_1 + \cdots + \alpha_d}
\widetilde{\Phi} (x - \alpha )
=0
\end{equation*}
and hence 
$\widetilde{\Phi}$ 
does not satisfy the condition (A). 



\begin{thebibliography}{99}
%





\bibitem{BT-2004}
\'A. B\'enyi and R. Torres,
{Almost orthogonality and a class of bounded bilinear
pseudodifferential operators},
Math. Res. Lett., \textbf{11}, (2004), 1--11.



\bibitem{boulkhemair 1995}
A. Boulkhemair,
{$L^2$ estimates for pseudodifferential operators},
Ann. Scuola Norm. Sup. Pisa Cl. Sci. (4), \textbf{22}, (1995), 155--183.


\bibitem{BGHH}
E.\ Buri\'ankov\'a, 
L.\ Grafakos, 
D.\ He, 
and 
P.\ Honz\'ik,  
{The lattice bump multiplier problem}, 
preprint. 


\bibitem{CM}
R.R. Coifman and Y. Meyer,
{Au del\`a des op\'erateurs pseudo-diff\'erentiels},
Ast\'erisque, \textbf{57}, (1978), 1--185.


\bibitem{fournier stewart 1985}
J.J.F. Fournier and J. Stewart,
{Amalgams of $L^p$ and $\ell^q$},
Bull. Amer. Math. Soc. (N.S.), \textbf{13}, (1985), 1--21. 

\bibitem{goldberg 1979} 
D. Goldberg, 
{A local version of real Hardy spaces}, 
Duke Math. J., \textbf{46}, (1979), 27--42.



\bibitem{GHS}
L. Grafakos, D. He, and 
L. Slav\'ikov\'a, 
$L^2 \times L^2 \rightarrow L^1$ boundedness criteria,  
Math. Ann., \textbf{376} (2020), 431--455. 


\bibitem{GK-Marcinkiewicz}
L. Grafakos and N. J. Kalton,
{The Marcinkiewicz multiplier 
condition 
for bilinear operators},
Studia. Math., \textbf{146}, (2001), 115--156.

\bibitem{GT}
L. Grafakos and R. Torres,
{Multilinear Calder\'on-Zygmund theory},
Adv. Math., \textbf{165}, (2002), 124--164.

%

\bibitem{holland 1975}
F. Holland,
{Harmonic analysis on amalgams of $L^p$ and $\ell^q$}, 
J. London Math. Soc. (2), \textbf{10}, (1975), 295--305.



\bibitem{KMT-arXiv}
T. Kato, A. Miyachi, and N. Tomita,
{Boundedness of bilinear pseudo-differential operators of 
$S_{0,0}$-type on $L^2 \times L^2$}, 
to appear in J. Pseudo-Differential Operators and Appl.,  
available at arXiv:1901.07237.


\bibitem{KMT-JMSJ}
T. Kato, A. Miyachi, and N. Tomita,
{Boundedness of multilinear pseudo-differential operators of 
$S_{0,0}$-type on $L^2$-based amalgam spaces}, 
to appear in J. Math. Soc. Japan,  
available at arXiv:1908.11641.


\bibitem{Kenig-Stein}
C.\ Kenig and E.\ M.\ Stein, 
{Multilinear estimates and fractional integration}, 
Math.\ Res.\ Lett.\ \textbf{6} (1999), 1--15. 


%



\bibitem{MT-2013}
A. Miyachi and N. Tomita,
{Calder\'on-Vaillancourt-type theorem for bilinear operators},
Indiana Univ. Math. J., \textbf{62}, (2013), 1165--1201.




\bibitem{Slavikova}
L. Slav\'{i}kov\'{a}, 
Bilinear Fourier multipliers and 
the rate of decay of their derivatives, 
J. Approx. Theory, \textbf{261} (2021), 105485. 


%
\end{thebibliography}
\end{document}